\documentclass{amsart}
\usepackage{amssymb}
\usepackage{mathtools}
\usepackage{microtype}
\usepackage[svgnames]{xcolor}
\usepackage[unicode,
  colorlinks=true,
  linktocpage=true,
  citecolor=OliveDrab,
  linkcolor=DarkMagenta,
  pdftitle={On partial groups of small order},
  pdfauthor={Philip Hackney}
  ]{hyperref}
\usepackage{tikz-cd}
\usepackage{comment}

\usetikzlibrary{calc}
\usetikzlibrary{shapes.geometric}

\usepackage{booktabs}
\usepackage[capitalise,noabbrev]{cleveref}

\usepackage{enumitem}

\newcommand{\fin}{\Upsilon}

\DeclareMathOperator{\id}{id}
\newcommand{\op}{\textup{op}}
\newcommand{\set}{\mathsf{Set}}

\newcommand{\sym}{\mathsf{Sym}}
\DeclareMathOperator{\sk}{sk}
\DeclareMathOperator{\cosk}{cosk}
\DeclareMathOperator{\aut}{Aut}
\newcommand{\lr}[1]{\langle #1 \rangle}

\DeclareMathOperator{\mat}{Mat}

\usepackage{mathrsfs}
\newcommand{\edgemap}{\mathscr{E}}
\newcommand{\bous}{\mathscr{B}}
\newcommand{\bswords}{\mathbf{S}}

\usepackage[T1]{fontenc}

\newtheorem{theorem}{Theorem}[section]
\newtheorem{proposition}[theorem]{Proposition}
\newtheorem{corollary}[theorem]{Corollary}
\newtheorem{lemma}[theorem]{Lemma}

\theoremstyle{definition}
\newtheorem{definition}[theorem]{Definition}
\newtheorem{example}[theorem]{Example}

\theoremstyle{remark}
\newtheorem{remark}[theorem]{Remark}

\AddToHook{env/proposition/begin}{\crefalias{theorem}{proposition}}
\AddToHook{env/corollary/begin}{\crefalias{theorem}{corollary}}
\AddToHook{env/lemma/begin}{\crefalias{theorem}{lemma}}
\AddToHook{env/definition/begin}{\crefalias{theorem}{definition}}
\AddToHook{env/example/begin}{\crefalias{theorem}{example}}
\AddToHook{env/remark/begin}{\crefalias{theorem}{remark}}
\apptocmd{\appendix}{\crefalias{section}{appendix}}{}{}


\begin{document}

\title{On partial groups of small order}

\author{Philip Hackney}
\address{Department of Mathematics, University of Louisiana at Lafayette}
\email{philip@phck.net} 
\urladdr{http://phck.net}

\thanks{
This work was supported by a grant from the Simons Foundation (\#850849, PH). 
The author was partially supported by the Louisiana Board of Regents through the Board of Regents Support fund LEQSF(2024-27)-RD-A-31.
}

\begin{abstract}
We explain a computer enumeration of all partial groups (in the sense of Chermak) of order at most 10.
An accompanying dataset contains a full list, consisting of 123,650 partial groups of order at most 9 and 178,937,003 partial groups of order 10; the paper itself contains a complete list of indecomposable partial groups of order at most 5. 
Inspection of the data led us to conjecture and then prove two results: that indecomposable partial groups of dimension two less than their order are precisely skeleta of groups of that order, and partial groups of (higher Segal) degree at most 2 are 2-coskeletal.
\end{abstract}
  
\date{July 7, 2026}

\maketitle

\section{Introduction}

Partial groups, and the well-behaved subclass of \emph{localities}, were introduced by Chermak in \cite{Chermak:FSL} as a part of his proof of the existence and uniqueness of centric linking systems, a problem posed in \cite{BrotoLeviOliver:HTFS}.
Localities are an encapsulation of the idea of doing $p$-local analysis in finite group theory \emph{without} the presence of an ambient finite group.
The other main implementation of this idea is via the fusion systems of Puig \cite{Puig:FC,AschbacherOliver:FS}, which are closely related to localities (see e.g.\ \cite{ChermakHenke:FSLAD,Henke:CPNSRL}).

Our focus in this paper is not on localities, but on general partial groups, which have excellent categorical properties \cite{Chermak:FL,Salati:LCGRPG,HackneyLynd:PGSSS}.
We anticipate a number of `locality-like' structured partial groups outside the bounds of $p$-local finite group theory, with the first examples being the punctured Weyl groups from \cite{HackneyLynd:HSSPG}.
Partial groups also have a nice extension theory \cite{BrotoGonzalez:ETPG}, may be considered as generalizations of the pregroups of Stallings \cite{LemoineMolinier:PGPRFS,Stallings:GT3DM}, and are flexible enough that each group is the group of automorphisms of some partial group \cite{DiazRamosMolinierViruel:PPG}.

Our core contribution is a complete enumeration of all partial groups of order at most 10 (see \cref{table partial group counts} for orders up to 9, and \cref{sec results}).
Two prior results make such an enumeration feasible. 
First, a finite partial group is indeed a finite structure \cite{HackneyMolinier:DPG} (it has only finitely many `nondegenerate' simplices), even though it is not part of the definition.
Second, every partial group sits inside a canonical maximal partial group determined by its underlying binary partial group \cite{HackneyLyndSalati:PGPG}, which gives the search a known ceiling. 
The enumeration is coupled with a collection of routines for working with finite partial groups \cite{partialgroups_jl}, written in the Julia programming language \cite{Julia-2017}.
Our aim is to enable experimentation: studying properties of individual partial groups, constructing examples exhibiting desired behaviors, or testing conjectures across a catalog.

Two results in particular were conjectured based on the enumerated data and then proved: indecomposable partial groups whose dimension is two less than their order are precisely the $(n{-}1)$-skeleta of groups (\cref{dim one less size}), and a partial group of degree at most 2 is automatically 2-coskeletal (\cref{thm degree 2}).
We hope our enumeration will allow others to similarly test hypotheses about partial groups against a substantial corpus.

\begin{table}
\centering
\caption{Nontrivial partial groups by order, free, dimension}
\label{table partial group counts}
\begin{tabular}{l c r r r r r r r r r}
\toprule
& \multicolumn{8}{c}{dimension} & \\
\cmidrule(lr){3-10}
order & f & 1 & 2 & 3 & 4 & 5 & 6 & 7 & 8 & total \\
\midrule
2 & 0 & 1 &   &   &   &   &   &   &   & 1\\
\addlinespace[4pt]
3 & 0 & 1 &   &   &   &   &   &   &   & 1\\
  & 1 & 1 & 1 &   &   &   &   &   &   & 2\\
\addlinespace[4pt]
4 & 0 & 1 & 1 & 1 &   &   &   &   &   & 3\\
  & 1 & 1 & 2 & 1 &   &   &   &   &   & 4\\
\addlinespace[4pt]
5 & 0 & 1 & 1 & 1 &   &   &   &   &   & 3\\
  & 1 & 1 & 4 & 1 &   &   &   &   &   & 6\\
  & 2 & 1 & 5 & 1 & 1 &   &   &   &   & 8\\
\addlinespace[4pt]
6 & 0 & 1 & 2 & 3 &   &   &   &   &   & 6\\
  & 1 & 1 & 12 & 19 & 1 & 1 &   &   &   & 34\\
  & 2 & 1 & 19 & 19 & 2 & 1 &   &   &   & 42\\
\addlinespace[4pt]
7 & 0 & 1 & 5 & 17 &   &   &   &   &   & 23\\
  & 1 & 1 & 27 & 72 & 1 & 1 &   &   &   & 102\\
  & 2 & 1 & 78 & 92 & 3 & 1 &   &   &   & 175\\
  & 3 & 1 & 60 & 39 & 7 & 1 & 1 &   &   & 109\\
\addlinespace[4pt]
8 & 0 & 1 & 13 & 396 & 116 & 37 & 1 & 1 &   & 565\\
  & 1 & 1 & 95 & 2328 & 586 & 113 & 1 & 1 &   & 3125\\
  & 2 & 1 & 437 & 2390 & 440 & 90 & 1 & 1 &   & 3360\\
  & 3 & 1 & 638 & 1505 & 419 & 60 & 3 & 2 &   & 2628\\
\addlinespace[4pt]
9 & 0 & 1 & 31 & 1169 & 116 & 37 & 1 & 1 &   & 1356\\
  & 1 & 1 & 446 & 12891 & 697 & 149 & 1 & 1 &   & 14186\\
  & 2 & 1 & 3610 & 24021 & 1092 & 192 & 1 & 1 &   & 28918\\
  & 3 & 1 & 10477 & 22102 & 1424 & 148 & 4 & 2 &   & 34158\\
  & 4 & 1 & 4581 & 11875 & 12937 & 5389 & 47 & 2 & 2 & 34834\\
\midrule
all  & &  24 & 20545 & 78943 & 17842 & 6220 & 61 & 12 & 2 & 123649 \\
\bottomrule
\end{tabular}
\end{table}

\subsection{Guide to the paper}
The next section includes basic information about partial groups that will be needed in the remainder of the paper, with special attention paid to symmetric sets of matrices.
\Cref{sec:pg via gen matrix} introduces the idea of a partial group freely generated by a matrix (or set of matrices), including in particular the case of a binary partial group generated by a single multiplication.
These ideas are used in \cref{sec algorithm}, which describes in detail our algorithm for enumerating partial groups.
We then turn to a summary of counts of partial groups in \cref{sec results}, which can be inferred from the complete dataset \cite{partial_groups_data}.

The last section of the main body of the paper is more specialized and concerns the relationship between higher Segal spaces and partial groups \cite{HackneyLynd:HSSPG,Dyckerhoff:CPOC}. 
We provide an algorithm for computing the degree of a partial group, and connect degree to coskeletality.

The first appendix consists of an explicit catalog of (indecomposable) partial groups of order at most 5. 
\Cref{sec appendix dim and order} explains the kinds of partial groups that can appear when dimension is two less than order.
Finally, \cref{sec data format} briefly explains the structure of the file formats for the dataset \cite{partial_groups_data}.

\subsection{Related work: enumeration of IP loops} 
The underlying partial binary multiplication in a partial group is a partial version of an algebraic structure called an \emph{inverse property loop} (IP loop), a non-associative version of groups which includes the Moufang loops.
In particular, partial groups which have all binary multiplications defined are supported on an underlying IP loop.
Computer enumerations of the IP loops were previously undertaken by Ali and Slaney \cite{SlaneyAli:GLIP,AliSlaney:CLIP,iploops}.
Our work recovers the same number of IP loops through order 10.
For instance there are 47 IP loops of order 10 among the 697,209 \emph{binary partial groups} of that order.

\subsection*{AI tool usage}
This paper was written entirely by the author.
Code implementing the described algorithms was developed with the assistance of an LLM-based coding tool (Claude Code), with the author directing development and verifying correctness. 
The algorithms and mathematical content are the author's own work.

\subsection*{Acknowledgements}
The author is grateful to Justin Lynd and Walker Stern for helpful discussions related to this project, and to Justin Lynd and Edoardo Salati for excellent suggestions on a previous draft.

\section{Mathematical Background}

We begin with necessary background, building on the point of view on partial groups from \cite{HackneyLynd:PGSSS}.
The original definition of partial groups is due to Chermak in \cite{Chermak:FSL}, and the interpretation as simplicial sets appears in work of González \cite{Gonzalez:ETPGL}.

\subsection{Symmetric sets and partial groups}

Let $\fin$ be a skeleton of the category of finite nonempty sets; for concreteness, we take as objects the sets $[n] = \{0,1,\dots, n\}$ for $n\geq 0$.
A \emph{symmetric set} is a functor $X\colon \fin^\op \to \set$ whose value on $[n]$ will be denoted by $X_n$ (the \emph{set of $n$-simplices}) and whose value on $\alpha \colon [m] \to [n]$ in $\fin$ will be denoted by $\alpha^* \colon X_n \to X_m$.
We write $\sym$ for the category of symmetric sets, whose morphisms are natural transformations.
Notice that $\aut([n])$ is the symmetric group on $n+1$ letters $\{0,1,\dots, n\}$, so $X_n$ naturally has a right $\Sigma_{n+1}$-action.

In particular, $X_1$ is canonically an \emph{involutive set}, via the nontrivial automorphism of $[1]$.
As the unique map $[1] \to [0]$ is split, the corresponding $X_0 \to X_1$ is injective; elements in its image, here called \emph{identities}, always are fixed points for the involution.

\begin{example}\label{example mat}
If $S$ is a set, let $\mat(S)_n$ denote the set of $(n+1)$-by-$(n+1)$ matrices with entries drawn from $S$, i.e. the set of functions $[n] \times [n] \to S$.
This naturally assembles into a symmetric set: given $\alpha \colon [m] \to [n]$ and a function $[n] \times [n] \to S$, we can precompose by $\alpha \times \alpha$ to obtain a function $[m] \times [m] \to S$.
In other words, if $M = (f_{ij})_{0\leq i,j \leq n} \in \mat(S)_n$, then $\alpha^*M = (g_{ab})$ where $g_{ab} = f_{\alpha(a), \alpha(b)}$.
\end{example}

\begin{example}
If $X$ is any symmetric set, then there is a canonical map $\mu_X \colon X \to \mat(X_1)$.
For a $0\leq i,j\leq n$, let $\epsilon_{ij} = \epsilon_{ij}^n \colon [1] \to [n]$ be given by $0,1\mapsto i,j$.
If $x\in X_n$, then $\mu_X(x) = (\epsilon_{ij}^*x)_{0\leq i,j \leq n}$.
Matrices in the image of $\mu_X$ have special properties: for one, they are skew-symmetric with respect to the involution on $X_1$, for another, elements along the diagonal are in the image of $X_0 \to X_1$.
\end{example}

\begin{definition}
Let $X$ be a symmetric set and $n\geq 1$.
\begin{itemize}
\item The \emph{Segal map} $\edgemap_n \colon X_n \to X_1^{\times n}$ is the composite
\[ \begin{tikzcd}
X_n \rar{\mu_X} & \mat(X_1)_n \rar{\text{superdiagonal}} &[+1cm] X_1^{\times n}
\end{tikzcd} \]
sending $x$ to $(\epsilon_{01}^*x, \epsilon_{12}^*x, \dots, \epsilon_{n-1,n}^*x)$.
\item The \emph{Bousfield-Segal map} $\bous_n \colon X_n \to X_1^{\times n}$ is the composite
\[ \begin{tikzcd}
X_n \rar{\mu_X} & \mat(X_1)_n \rar{\text{tail of top row}} &[+1cm] X_1^{\times n}
\end{tikzcd} \]
sending $x$ to $(\epsilon_{01}^*x, \epsilon_{02}^*x, \dots, \epsilon_{0n}^*x)$.
\end{itemize}
\end{definition}
As explained by Grothendieck \cite[4.1]{Grothendieck:TCTGA3}, the categories of groups and groupoids embed into the category of symmetric sets.
A symmetric set is isomorphic to the nerve of a group if and only if it is \emph{reduced} (i.e. $X_0$ is a one-point set) and the Segal maps are bijections for all $n$.
In this case the $n$-ary group multiplication is given by the span
\begin{equation}\label{diag span} 
\begin{tikzcd}[sep=large]
X_1 \times \underset{n}{\cdots} \times X_1 & X_n \lar["\edgemap_n"', "\cong"] \rar["x\mapsto \epsilon_{0n}^*x"] & X_1.
\end{tikzcd} \end{equation}
This leads us to a definition of partial group \cite{Chermak:FSL}, based on Corollary 4.7 of \cite{HackneyLynd:PGSSS}.

\begin{definition}
A symmetric set is a \emph{partial group} if it is reduced and if the Segal map $\edgemap_n$ (alternatively, the Bousfield-Segal map $\bous_n$) $X_n \to X_1^{\times n}$ is injective for each $n\geq 1$.
The category of partial groups is the full subcategory of $\sym$ on the partial groups.
\end{definition}

For a partial group $X$, the span \eqref{diag span} endows $X_1$ with a partially-defined $n$-ary multiplication $X_1^{\times n} \nrightarrow X_1$.
Theorem 3.6 of \cite{HackneyLynd:PGSSS} implies that $\edgemap_n$ is injective for all $n\geq 2$ if and only if $\bous_n$ is injective for all $n\geq 2$. 
A \emph{spiny symmetric set} is a symmetric set where all (Bousfield-)Segal maps are injections.
The symmetric set $\mat(S)$ is not often spiny.

If $X$ is any spiny symmetric set, then $\mu_X \colon X \to \mat(X_1)$ is a monomorphism.
We often use $\mu_X$ to write $n$-simplices as $(n+1)$-by-$(n+1)$ matrices $(f_{ij})$.
Other papers in the partial groups literature consider `multipliable words' in $X_1$ as the primary representation of $n$-simplices; such a word $w = (g_1, \dots, g_n)$ then corresponds to $x\in X_n$ satisfying $\edgemap_n(x) = w$.

\subsection{Skeleta and Dimension}\label{skeleta and dimension}

An $n$-simplex $x$ in a symmetric set $X$ is \emph{nondegenerate} if there does not exist a pair $(y,\rho)$ consisting of an $m$-simplex $y$ and a noninvertible surjection $\rho \colon [m] \twoheadrightarrow [n]$ such that $x = \rho^*y$.
If $X$ is a spiny symmetric set, $x\in X_n$, and $(x_{ij}) = \mu_X(x)$, then the following are equivalent (see \cite[Lemma 2.13]{HackneyLynd:HSSPG} and \cite[Lemma 7]{HackneyMolinier:DPG}):
\begin{enumerate}
\item $x$ is nondegenerate.
\item If $x_{ij}$ is an identity (i.e. in the image of $X_0 \to X_1$), then $i=j$.
\item No row of $(x_{ij})$ contains a repeated element.
\item $\bous_n(x)$ consists of distinct, nonidentity elements.
\end{enumerate}

Note: a degenerate element may have all nonidentity elements on its superdiagonal.

\begin{definition}
Let $n \geq -1$ be an integer. 
A symmetric set is \emph{$n$-skeletal} if all $m$-simplices for $m > n$ are degenerate.
The \emph{$n$-skeleton} of a symmetric set $X$, denoted $\sk_n X$, is the smallest symmetric subset of $X$ containing all $k$-simplices for $k$ at most $n$.
A symmetric set is said to have dimension $n$ if it is $n$-skeletal but not $(n{-}1)$-skeletal.
\end{definition}

If $X$ and $Y$ are finite dimensional partial groups, then their coproduct $X \vee Y$ has $\dim(X\vee Y) = \max(\dim(X), \dim(Y))$.

The $n$-skeletal symmetric sets may be regarded as presheaves on a finite category: let $\fin_{\leq n} \subset \fin$ be the full subcategory whose objects are $[k]$ for $0\leq k \leq n$, and write $\iota$ for the inclusion functor.
Generally $\sk_n X = \iota_!\iota^* X$ is the left Kan extension of the restriction of $X$ to $\fin_{\leq n}^\op$.
Moreover if $Z \colon \fin_{\leq n}^\op \to \set$ is any object which is spiny, then its left Kan extension $\iota_! Z \colon \fin^\op \to \set$ is also spiny \cite[Theorem 9]{HackneyMolinier:DPG}. 

The following appears in \cite{HackneyMolinier:DPG} (Corollaries 15 and 19).

\begin{theorem}\label{hm thm}
If $X$ is a partial group of order $n+1$, then $X$ is $n$-skeletal; it is isomorphic to the nerve of a group if and only if its dimension is $n$.
\end{theorem}

We now turn to the case of partial groups of dimension at most 2.

\begin{definition}\label{def bin partial group}
A \emph{binary partial group} (\emph{BPG}) is a unital partial magma $P$ along with a map $\dagger \colon P\to P$ such that if $ab$ is defined then $a^\dagger(ab) = b$ and $(ab)b^\dagger = a$.
\end{definition}
We emphasize that the multiplication on the left side of the equations are also defined whenever $ab$ is defined.
These were first introduced in \cite[\S4]{Baer:FSGG} (under a different name), and were later considered in \cite[II.B]{Tamari:PAMPMG}.
A BPG in which the multiplication is totally defined is the same data as an inverse property loop (IP loop) \cite[VII.1]{Bruck:SBS}.

Binary partial groups were shown to be equivalent to 2-skeletal partial groups in \cite{HackneyLyndSalati:PGPG}, and we will use this identification freely below.
If $P$ is a BPG, then a multiplication fact $ab=c$ in $P$ corresponds to a 2-simplex whose matrix form is
\[
\begin{bmatrix}
e & b & c \\
b^\dagger & e & a \\
c^\dagger & a^\dagger & e
\end{bmatrix}.
\]
There is thus a $\Sigma_3$-action on the set of multiplications in a BPG, and the equations in \cref{def bin partial group} can be regarded as the action by the generators $(12)$ and $(01)$ of $\Sigma_3 = \aut([2])$.
For instance, $(12)$ has the following effect on a matrix:
\[
\begin{bmatrix}
e & b & c \\
b^\dagger & e & a \\
c^\dagger & a^\dagger & e
\end{bmatrix}
\mapsto 
\begin{bmatrix}
e & c & b \\
c^\dagger & e & a^\dagger \\
b^\dagger & a & e
\end{bmatrix},
\]
and so sends the equation $ab=c$ to the equation $a^\dagger c = b$, i.e.\ $a^\dagger(ab)=b$.
A BPG may also be represented by a (partial) Cayley table, where some entries may be blank (see \cref{sec tiny order} for examples).

Given a BPG $P$ the authors of \cite{HackneyLyndSalati:PGPG} construct a maximal partial group $BP$ whose 2-skeleton corresponds to $P$.
The set $BP_n \subseteq P^{\times n}$ consists of those words such that for every parenthesization, the iterated multiplication in $P$ is defined, and all of these iterated multiplications yield the same element of $P$.
The equivalence of BPGs with 2-skeletal partial groups sends $P$ to $\sk_2 BP$.
As was known to Hackney--Lynd--Salati, symmetric sets of the form $BP$ are precisely the \emph{2-coskeletal} partial groups (see \cref{deg 2 partial groups}).

\section{Partial groups via generating matrices}\label{sec:pg via gen matrix}
In this section we describe when and how a single matrix, or a collection of matrices, generates a partial group (or more generally a spiny symmetric set).
We identify the smallest such pieces in the 2-dimensional case, the \emph{atomic BPGs} which are the BPGs freely generated by a single nontrivial multiplication.

\subsection{From matrices to spiny symmetric sets}\label{sec matrices to SSS}
A basic question is when the symmetric subset of $\mat(S)$ generated by a matrix $M\in \mat(S)_n$ is spiny. 

We work over an involutive, reflexive, directed graph $\Gamma$.
This is the same as a presheaf on $\fin_{\leq 1}$, that is, a diagram of sets of shape
\[ \begin{tikzcd}
\Gamma_0 \rar["\id" description] & \Gamma_1 \lar[shift left=2, "t"] \lar[shift right=2, "s"'] \arrow[loop right, out=30, in=330, looseness=4, "\dagger "]
\end{tikzcd} \]
such that $(e^\dagger)^\dagger = e$ and $t(e^\dagger) = s(e)$ for all $e\in \Gamma_1$, and $s(\id_v) = v = t(\id_v)$ and $(\id_v)^\dagger = \id_v$ for all $v\in \Gamma_0$.

\begin{definition}\label{def suitable}
If $\Gamma$ is such a graph we will say that a matrix $M = (a_{ij})\in \mat(\Gamma_1)_n$ is \emph{suitable} if the following conditions hold:
\begin{enumerate}
\item $M$ is skew-symmetric: $a_{ij} = a_{ji}^\dagger$. \label{suitable skew}
\item All diagonal entries of $M$ are identities: $a_{ii} = \id_v$. \label{suitable diag}
\item All entries on a given row have the same source: $s(a_{ij}) = s(a_{ik})$. \label{suitable sources}
\item If there are $i,j,k$ and $p,q,r$ such that $a_{ij} = a_{pq}$ and $a_{ik} = a_{pr}$, then $a_{jk} = a_{qr}$.\label{suitable BS cond}
\end{enumerate}
\end{definition}

If $X$ is any symmetric set and $\Gamma$ is its restriction to $\fin_{\leq 1}$, then the first three conditions hold for any $M\in \mat(X_1)$ in the image of $\mu_X$.
The last condition holds when $X$ is spiny.

Any matrix $M \in \mat(\Gamma_1)_n$ generates a symmetric subset $X \subseteq \mat(\Gamma_1)$, where $X_k$ consists of those matrices $\alpha^*M = (a_{\alpha(i),\alpha(j)})$ where $\alpha\colon [k] \to [n]$ is a function.
In other words, $X$ is the smallest symmetric subset of $\mat(\Gamma_1)$ containing $M$.
If $M$ is suitable, then so is every element of $X$.
In this case we write $\lr{M}_0 \subseteq \Gamma_0$ for the collection of those $v$ such that $\id_v$ appears on the diagonal of $M$, and $\lr{M}_1 \subseteq \Gamma_1$ for the set of elements appearing in $M$.
The assignments $v \mapsto [\id_v]$ and $(a\colon v\to w) \mapsto \begin{bsmallmatrix}
\id_v & a \\
a^\dagger & \id_w
\end{bsmallmatrix}$ give bijections $\lr{M}_k \cong X_k$ for $k =0,1$.
We set $\lr{M}_k = X_k$ for $k\geq 2$, and identify $\lr{M}$ and $X$ as symmetric sets.
We'll use the following generally useful lemma to show that $X\cong \lr{M}$ is spiny.

\begin{lemma}\label{spininess for mat subobjects}
Suppose $Y$ is a symmetric set and $\mu_Y \colon Y \to \mat(Y_1)$ is injective. 
If either $\edgemap_2$ or $\bous_2$ from $Y_2$ to $Y_1 \times Y_1$ is injective, then $Y$ is spiny.
\end{lemma}
\begin{proof}
Below we write $\mu_Y(y) = (a_{ij})$ and $\mu_Y(y') = (a'_{ij})$ for the matrix forms of simplices $y,y'$ in $Y$.

Suppose $\bous_2$ is injective, and $\bous_n(y) = \bous_n(y')$, i.e.\ $a_{0i} = a'_{0i}$ for $i=1, \dots, n$.
If $\alpha_{ij} \colon [2] \to [n]$ is given by $0,1,2 \mapsto 0,i,j$, then $\bous_2(\alpha_{ij}^* y) = (a_{0i},a_{0j})$ is equal to $\bous_2(\alpha_{ij}^* y') = (a_{0i}', a_{0j}')$, hence $\alpha_{ij}^*y = \alpha_{ij}^*y'$. This implies $a_{ij} = a_{ij}'$.

Suppose $\edgemap_2$ is injective. 
For $n > 2$, inductively assume $\edgemap_{n-1}$ is injective. If $\edgemap_n(y) = \edgemap_n(y')$, then $\edgemap_{n-1}(d_iy) = \edgemap_{n-1}(d_iy')$ for $i=0,n$, and hence $d_iy = d_iy'$ for $i=0,n$. 
So $a_{ij} = a'_{ij}$ unless $\{i,j\} = \{0,n\}$.
The only question is to compare $a_{0n}$ and $a'_{0n}$.
The map $\alpha_{n-1,n}$ does the trick: $\edgemap_2(\alpha_{n-1,n}^*y) = (a_{0,n-1}, a_{n-1,n}) = \edgemap_2(\alpha_{n-1,n}^*y')$, so $\alpha_{n-1,n}^*y = \alpha_{n-1,n}^*y'$, and hence $a_{0n} = a'_{0n}$. 
\end{proof}

\begin{theorem}\label{suitable to spiny}
If $M \in \mat(\Gamma_1)_n$ is suitable, then $\lr{M}$ is spiny.
\end{theorem}
\begin{proof}
Condition \eqref{suitable BS cond} is exactly stating that the Bousfield--Segal map $\bous_2$ is injective for $\lr{M}$, so the result follows from \cref{spininess for mat subobjects}.
More explicitly, let $\alpha^*M, \beta^*M \in \lr{M}_2$ be two matrices with $\alpha$ sending $0,1,2$ to $i,j,k$ and $\beta$ sending $0,1,2$ to $p,q,r$.
Then
\[
  \bous_2(\alpha^*M) = \bous_2 \left(
  \begin{bmatrix}
    \id & a_{ij} & a_{ik} \\
    a_{ji} & \id & a_{jk} \\
    a_{ki} & a_{kj} & \id
  \end{bmatrix}
  \right) = (a_{ij}, a_{ik})
\]
and similarly $\bous_2(\beta^*M) = (a_{pq}, a_{pr})$.
Noting that $\alpha^*M$ and $\beta^*M$ are uniquely determined by their upper triangles by the first three conditions of suitability, we see that $\bous_2$ is injective if and only if \eqref{suitable BS cond} holds.
\end{proof}

\begin{remark}
Given a set $S$, one can begin with a raw matrix $M\in \mat(S)$ and ask whether $M$ generates a spiny symmetric set.
In this case, let $\Gamma_0 \subseteq S$ be the set of elements appearing on the diagonal of $M$, and $\Gamma_1 \subseteq S$ be the set of elements appearing in $M$.
We attempt to define an involution on $\Gamma_1$ by sending $a_{ij}$ to $a_{ji}$, a source map $s\colon \Gamma_1 \to \Gamma_0$ sending $a_{ij}$ to $a_{ii}$, and a target map $t\colon \Gamma_1 \to \Gamma_0$ sending $a_{ij}$ to $a_{jj}$.
Of course any of these three functions may fail to be well-defined!
But if they're well-defined then $\Gamma$ constitutes an involutive, reflexive, directed graph. 
Only condition \eqref{suitable BS cond} then needs to be checked in order for $M$ to be suitable in $\mat(\Gamma_1)$.
\end{remark}

\begin{remark}
As our focus on this paper is on partial groups, we may restrict our attention to (involutive, reflexive, directed) graphs where $\Gamma_0$ is a one-point set.
These are the same thing as pointed involutive sets.
Notice that Condition \eqref{suitable sources} of suitability is automatic in this case.
\end{remark}

\begin{definition}
Let $\Gamma$ be a graph and let $M = (a_{ij})$ and $M' = (b_{ij})$ be two suitable matrices over $\Gamma$.
We say that $M$ and $M'$ are \emph{compatible} if whenever there are $i,j,k$ and $p,q,r$ such that $a_{ij} = b_{pq}$ and $a_{ik} = b_{pr}$, then $a_{jk} = b_{qr}$.
\end{definition}

\begin{proposition}\label{prop pairwise compatible}
Let $\Gamma$ be a graph and let $S = \{M_\alpha\}$ be a set of suitable matrices over $\Gamma$.
Then $\bigcup \lr{M_\alpha}$ is spiny if and only if the elements of $S$ are pairwise compatible.
\end{proposition}
\begin{proof}
Each individual $\lr{M_\alpha}$ is spiny by \cref{suitable to spiny}.
A direct application of \cref{spininess for mat subobjects} gives that the union is spiny, as pairwise compatibility guarantees $\bous_2$ is injective.
\end{proof}

\subsection{Atomic binary partial groups}\label{ss abpgs}
A BPG is said to be \emph{atomic} if it is generated by a single nontrivial multiplication.
More precisely, we are looking at those BPGs which correspond to partial groups generated by a single nondegenerate 3-by-3 matrix
\begin{equation*}\label{generic 3-simplex}
  M = \begin{bmatrix}
    \id & b & c \\
    b^\dagger & \id & a \\
    c^\dagger & a^\dagger & \id
  \end{bmatrix}
\end{equation*}
representing $ab=c$. 
For $M$ to generate a partial group, we need $M$ to be suitable, i.e.\ to satisfy Condition \eqref{suitable BS cond} above.
Here, nondegeneracy means that $a,b,c$ are drawn from an involutive set consisting of nonidentity elements, or, alternatively, that there are no repeated elements in any row.
Up to isomorphism, there are seven such matrices:
\[
\begin{gathered}
\begin{bmatrix} 
  0  & 2 & 3\\
 -2  & 0 & 1\\
 -3  & -1 & 0\\
\end{bmatrix}
\qquad
\begin{bmatrix} 
  0  & 2 & 3\\
 -2  & 0 & 1\\
 3 & -1 & 0\\
\end{bmatrix}
\qquad
\begin{bmatrix} 
  0  & 2 & 3\\
 2  & 0 & 1\\
 3 & -1 & 0\\
\end{bmatrix}
\qquad
\begin{bmatrix} 
  0  & 2 & 3\\
  2  & 0 & 1\\
  3  &  1 & 0\\
\end{bmatrix}
\\
\begin{bmatrix} 
  0  & 1 & 2\\
 -1  & 0 & 1\\
 -2 & -1 & 0\\
\end{bmatrix}
\qquad
\begin{bmatrix} 
  0  & 1 & 2\\
 -1  & 0 & 1\\
  2 & -1 & 0\\
\end{bmatrix}
\qquad
\begin{bmatrix} 
  0  & 1 & -1\\
 -1  & 0 & 1\\
  1 & -1 & 0\\
\end{bmatrix}
\end{gathered}
\]
We'll name the BPGs these matrices generate after the \emph{type} of the underlying involutive set: $A^{3,0},A^{2,1},A^{1,2},A^{0,3}$ and $A^{2,0},A^{1,1},A^{1,0}$.
The first, $A^{3,0}$ is the free partial group on a length two word (denoted $\mathfrak{F}^2$ in \cite[Example 5.4]{HackneyLynd:PGSSS}), and the last is the group $C_3 \cong A^{1,0}$.

\begin{theorem}
This list of seven atomic BPGs is exhaustive.
\end{theorem}
\begin{proof}
The 3-by-3 matrices we are considering have $a\neq c$, $b\neq c$, and $a\neq b^\dagger$.
There are up to eight possible involutive sets (corresponding to whether $x=x^\dagger$ for $x=a,b,c$) $\{a,b,c,a^\dagger,b^\dagger,c^\dagger\}$ to consider having the additional properties $a\neq b$ and $a\neq c^\dagger \neq b$, yet we see only four matrices on the top line. The atomic BPG $A^{1,2}$ has among its 2-simplices
\[
\begin{bmatrix}
  0 & 1 & 2\\
 -1 & 0 & 3\\
  2 & 3 & 0
\end{bmatrix}
\quad \& \quad
\begin{bmatrix}
 0 & 2 & 1 \\
 2 & 0  & 3 \\
 -1 & 3  & 0
\end{bmatrix}
\]
so in fact this covers all of the cases where $x \neq x^\dagger$ for exactly one $x\in \{a,b,c\}$.
Meanwhile, $A^{2,1}$ has among its 2-simplices
\[
\begin{bmatrix}
  0 & 1 & -2\\
 -1 & 0 &  3\\
  2 & 3 &  0
\end{bmatrix}
\quad \& \quad
\begin{bmatrix}
 0  & 3 & -1\\
 3 &  0 &  2\\
 1 & -2 &  0
\end{bmatrix}
\]
so covers all of the cases where $x=x^\dagger$ for exactly one $x\in \{a,b,c\}$.

There is also the possibility of having $a=b$ (when $a\neq a^\dagger$), in which case $c$ could be a new element, or could be $a^\dagger$ itself.
In the former case, there are two possibilities: $c=c^\dagger$ and $c\neq c^\dagger$, so the bottom row represents all possibilities with $a=b$.
Finally, one might wonder about what happens when $c$ is $a^\dagger$ or $b^\dagger$.
But the four relevant cases are already handled by the bottom row, as
$A^{1,1}$ contains 
\[
\begin{bmatrix}
  0  & 2 & -1\\
  2  & 0 &  1\\
  1 & -1 &  0\\
\end{bmatrix}
\quad \& \quad
\begin{bmatrix}
   0 & 1 & -1\\
  -1 & 0 &  2\\
   1 & 2 &  0 
\end{bmatrix}
\]
and $A^{2,0}$ contains 
\[
\begin{bmatrix}
  0 & 2 &  1 \\
 -2 & 0 & -1 \\
 -1 & 1 &  0
\end{bmatrix}
\quad \& \quad
\begin{bmatrix}
  0 & -1 & 1 \\
  1 &  0 & 2 \\
 -1 & -2 & 0
\end{bmatrix}.
\]
We conclude that $A^{3,0},A^{2,1},A^{1,2},A^{0,3},A^{2,0},A^{1,1},A^{1,0}$ is the full list.
\end{proof}

\section{Enumerating partial groups of small order}
\label{sec algorithm}
An approach to enumerating all partial groups is the following:
\begin{enumerate}
\item Start with a fixed $n$-dimensional partial group $X$.\label{step 1}
\item Find all $(n{+}1)$-dimensional partial groups $Y$ with the property that $\sk_n Y = X$.\label{step 2}
\item Reduce this collection by modding out by isomorphisms fixing the $n$-skeleton: pick one representative from each isomorphism class.\label{step 3}
\item Increase $n$ and repeat this process for each representative $Y$ appearing in the previous step.\label{step 4}
\end{enumerate}
If we start with the unique 0-dimensional partial group, then in the limit we will obtain all partial groups up to isomorphism.

We wish instead to enumerate all partial groups of a fixed order.
For order 1, there is only the trivial group.
Otherwise, we start with a 1-dimensional partial group in \eqref{step 1}, which is more or less the same thing as an involutive set.

\subsection{The involutive set}\label{subsec inv set}
Given a partial group $X$, one can consider the involutive set $I = X_1 \setminus \id$ of nonidentity edges.
Up to isomorphism, a finite involutive set is fully specified by pair $(\mathtt{free},\mathtt{fixed}) = (a,b)$ of nonnegative integers, where $a$ represents the number of free orbits under the involution, and $b$ represents the number of fixed orbits.
For concreteness, we always suppose that $X_1$ is a set of integers of the form \[ X_1 = [-a,a+b] = \{0, \pm 1, \dots, \pm a, a+1, \dots, a+b\}\] 
where $0$ is the identity, the involution acts via negation on $[-a,a]$ and fixes $[a+1, a+b]$.

The automorphism group is the product of a hyperoctahedral group (group of signed permutation matrices) and a symmetric group: \[ \aut(I) = \aut(X_1) = \aut([-a,a]) \times \aut([a+1,a+b]) \cong (C_2 \wr \Sigma_a) \times \Sigma_b. \] 
(Here $\aut([-a,a])$ and $\aut(X_1)$ are groups of \emph{pointed} automorphisms -- automorphisms preserving the basepoint $0$).
Altogether $\aut(I)$ has $a! \cdot 2^a \cdot b!$ elements.

\subsection{Binary partial groups with a given underlying involutive set}\label{subsec BPG enum}
In this section we describe an algorithm for finding all binary partial groups whose nonidentity elements come from a fixed involutive set $I$ (briefly \emph{the BPGs on $I$}).
An \emph{atomic BPG on $I$} is an atomic BPG in the sense of \S\ref{ss abpgs} whose matrices have off-diagonal entries chosen from $I$.
In other words, each is specified by a single nontrivial multiplication.

Given a 2-skeletal partial group $Y$, we may decompose $Y$ as 
\[
  Y = (\sk_1 Y) \cup \lr{M_1} \cup \cdots \cup \lr{M_p}
\]
where the $M_i$ are nondegenerate 2-simplices of $Y$, one for each $\Sigma_3$-orbit (alternatively: $M_i \notin \lr{M_j}$ for $i\neq j$).
Each of the $\lr{M_i}$ is an atomic BPG.
We take this as our starting point for enumerating 2-dimensional partial groups with a fixed 1-skeleton: take the union of $\sk_1Y$ with several $\lr{M}$.
However it is not always possible to do so: we may have two $M,M' \in \mat(Y_1)$ such that the symmetric subset generated by $M$ and $M'$ is not a partial group (see \cref{prop pairwise compatible}).

\begin{definition}
Fix an involutive set $I$. 
\begin{itemize}
\item Suppose $P$,$P'$ are two BPGs on $I$. We say $P$ and $P'$ are \emph{compatible} if whenever the product of $a,b\in I$ is defined in both $P$ and $P'$, these products are equal.
\item The \emph{compatibility graph} of $I$ is the (simple) graph which has as its nodes the atomic BPGs on $I$, and edges exactly between the compatible atomic BPGs.
\end{itemize}
\end{definition}

\begin{proposition}
The binary partial groups on $I$ are in bijection with the set of cliques in the compatibility graph of $I$. \qed
\end{proposition}

\begin{figure}
\begin{tikzpicture}[every node/.style={draw, rounded corners=4pt, inner sep=3pt, font=\small, minimum width=1.65cm}, thick]     

  \node (A1) at (-2,0.75) {$1\cdot1=-2$};  
  \node (A2) at (0,0.75) {$1\cdot1=-1$};
  \node (A3) at (2,0.75) {$1\cdot1=2$};

  \node (B1) at (-2,-0.75) {$2\cdot2=1$};
  \node (B2) at (0,-0.75) {$2\cdot2=-2$};
  \node (B3) at (2,-0.75) {$2\cdot2=-1$};

  \draw (A1) -- (B1);
  \draw[shorten >=-0.75pt, shorten <=-0.75pt] (B1) -- (A2);
  \draw[shorten >=-0.75pt, shorten <=-0.75pt] (A2) -- (B3);
  \draw (B3) -- (A3);
  \draw[shorten >=-0.75pt, shorten <=-0.75pt] (A1) -- (B2);
  \draw[shorten >=-0.75pt, shorten <=-0.75pt] (B2) -- (A3);
  \draw (A2) -- (B2);

\end{tikzpicture}
\caption{Compatibility graph of atomic BPGs for $(2,0)$}\label{fig compat20}
\end{figure}
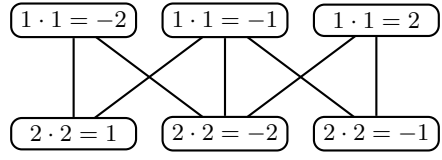

\begin{example}\label{ex compat20}
\Cref{fig compat20} is the compatibility graph among all of the atomic BPGs on the involutive set $I = \{\pm 1, \pm 2\}$.
The labels on the nodes are a generating multiplication for the associated atomic BPG. 
For instance the atomic BPG in the top center has two nontrivial multiplications: $1\cdot 1 = -1$, as well as $-1 \cdot -1 = 1$.
There are six elements in the $\Sigma_3$-orbit of $1\cdot 1 = -2$ -- this node also includes the multiplications $-1 \cdot -2 = 1$, $-2 \cdot -1 = 1$, $-1 \cdot -1 = 2$, $1 \cdot 2 = -1$ and $2 \cdot 1 = -1$.
This node is not connected to $2\cdot 2 = -1$, as the orbit of $2\cdot 2 = -1$ contains $-2 \cdot -1 = 2$.
\end{example}

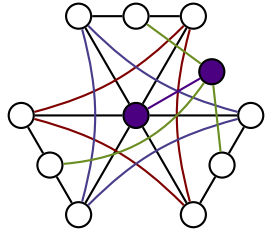
\begin{figure}

\begin{tikzpicture}[thick]

  \node[regular polygon, regular polygon sides=6,
        minimum size=3cm, draw=none] (hex) at (0,0) {};

  \foreach \i in {1,...,6}
    \node[circle, draw] (v\i) at (hex.corner \i) {}; 

  \foreach \i/\j in {1/2, 3/4, 5/6} {
    \coordinate (mid\i\j) at ($(v\i)!0.5!(v\j)$);
    \node[circle, draw] (m\i\j) at (mid\i\j) {}; 
  }

  \node[circle, draw, fill=Indigo] (center) at (0,0) {};

\node[circle, draw, fill=Indigo] (loose) 
  at ($(mid12)!0.5!(mid56)!-0.5cm!(center)$) {};

  \foreach \i in {1,...,6}
    \draw (center) -- (v\i);

  \foreach \i/\j in {1/2, 3/4, 5/6} {
    \draw (m\i\j) -- (v\i);
    \draw (m\i\j) -- (v\j);
  }

\draw[Maroon, thick] (v1) to[bend left=15] (v3);
\draw[Maroon, thick] (v5) to[bend left=15] (v1);
\draw[Maroon, thick] (v3) to[bend left=15] (v5);

\draw[DarkSlateBlue, thick] (v2) to[bend left=15] (v4);
\draw[DarkSlateBlue, thick] (v6) to[bend left=15] (v2);
\draw[DarkSlateBlue, thick] (v4) to[bend left=15] (v6);

\draw[OliveDrab, thick] (loose) -- (m12);
\draw[OliveDrab, thick] (loose) to[bend left=25] (m34);
\draw[OliveDrab, thick] (loose) -- (m56);
\draw[Indigo, thick] (center) -- (loose);

\end{tikzpicture}
\caption{Compatibility graph of atomic BPGs for $(1,3)$}\label{fig compat13}
\end{figure}

\begin{example}
\Cref{fig compat13} is the compatibility graph among all of the atomic BPGs on the involutive set $I = \{\pm 1, 2, 3, 4\}$.
We've arranged 9 of the 11 nodes on a regular hexagon. 
The six corner nodes correspond to $1\cdot a = b$ for $a\neq b$ in $\{2,3,4\}$, while the three nodes on hexagon edges correspond to $1\cdot 1 = a\in \{2,3,4\}$.
The center node is $1\cdot 1 = -1$, and the off-center node is $2\cdot 3 = 4$.
The maximal cliques have size 4.
\end{example}

We now have excellent combinatorial control over the collection of BPGs on $I$.
The standard Bron--Kerbosch algorithm for finding cliques in a graph may be used to find the full collection of BPGs \cite{BronKerbosch:A457}.
This means we've completed Step~\ref{step 2}.
However, we are only interested in enumerating partial groups \emph{up to isomorphism}.
As $I$ is fixed, this means we should quotient by the action of $\aut(I)$ on the collection of BPGs that we found.
Namely, the action of $\aut(I)$ on the set of matrices induces an action of $\aut(I)$ on the set of BPGs on $I$, and also induces an action of $\aut(I)$ on the compatibility graph of $I$.

For memory and speed reasons, it is best to compute the action of $\aut(I)$ on the set of cliques in the compatibility graph, choose a representative in each $\aut(I)$-orbit, and then form the associated BPG. 
(Rather than forming the collection of all BPGs on $I$, and then deduplicating using the $\aut(I)$-action.)

\begin{example}
For the $(2,0)$ case from \cref{ex compat20} and \cref{fig compat20} we see that there are seven cliques of size 2, six cliques of size 1, and the empty clique, so fourteen in total. It turns out there are six elements in the quotient, and one can compute this directly from the picture.
The order 8 group $\aut(I)$ is generated by two elements.
The generator $(1,2) \mapsto (2,1)$ exchanges the two middle vertices and swaps the corner vertices of the rectangle along diagonals.
The generator $(1,2)\mapsto (1,-2)$ on corners swaps horizontally across the central edge, and fixes the two central vertices.
So the size 2 cliques partition into 3 isomorphism classes, and the size 1 cliques partition into 2 isomorphism classes, as depicted in \cref{fig compat20 quotient}.
Along with the empty clique, this gives us the six BPGs on $I$.
\end{example}

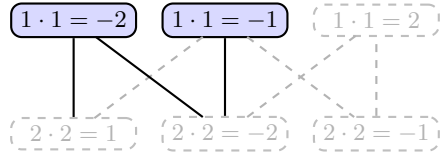
\begin{figure}
\begin{tikzpicture}[every node/.style={draw, rounded corners=4pt, inner sep=3pt, font=\small, minimum width=1.65cm}, thick]     

  \node[fill=blue!15] (A1) at (-2,0.75) {$1\cdot1=-2$};  
  \node[fill=blue!15] (A2) at (0,0.75) {$1\cdot1=-1$};
  \node[draw=gray!50, dashed, text=gray!60] (A3) at (2,0.75) {$1\cdot1=2$};

  \node[draw=gray!50, dashed, text=gray!60] (B1) at (-2,-0.75) {$2\cdot2=1$};
  \node[draw=gray!50, dashed, text=gray!60] (B2) at (0,-0.75) {$2\cdot2=-2$};
  \node[draw=gray!50, dashed, text=gray!60] (B3) at (2,-0.75) {$2\cdot2=-1$};

  \draw (A1) -- (B1);
  \draw[shorten >=-0.75pt, shorten <=-0.75pt, dashed, gray!60] (B1) -- (A2);
  \draw[shorten >=-0.75pt, shorten <=-0.75pt, dashed, gray!60] (A2) -- (B3);
  \draw[dashed, gray!60] (B3) -- (A3);
  \draw[shorten >=-0.75pt, shorten <=-0.75pt] (A1) -- (B2);
  \draw[shorten >=-0.75pt, shorten <=-0.75pt, dashed, gray!60] (B2) -- (A3);
  \draw (A2) -- (B2);

\end{tikzpicture}
\caption{Representatives for isomorphism classes of cliques}\label{fig compat20 quotient}
\end{figure}

\begin{remark}[Implementation of BPGs]\label{bpg implementation}
In \cite{partialgroups_jl} a BPG is comprised of two pieces of data: an involutive set (effectively a pair (\texttt{free},\texttt{fixed}) of nonnegative integers) and a partial Cayley table.
The latter only records nontrivial multiplications, i.e.\ omits $0a = a = a0$ and $aa^\dagger = 0$. 
In particular, the trivial BPG on $I$ has an empty Cayley table.
\end{remark}

The task now for each BPG $P$ in our list is to compute all symmetric sets $X$ with $\sk_2 BP \subseteq X \subseteq BP$.
We have $\aut(P) = \aut(BP)$, and we only want to find a single $X$ in each $\aut(P)$-orbit.
We compute extensions one dimension at a time.

\subsection{Increasing the dimension}
Let $n\geq 3$, and fix a partial group $X$ with $\dim(X) = n-1$.
We explain how to find all possible $n$-dimensional extensions of $X$.
This process is more uniform than the transition from dimension $1$ to dimension $2$.

\begin{remark}[Implementation of partial groups]\label{rmk implementation}
Partial groups are stored as a pair consisting of a BPG (as in \cref{bpg implementation}) and an array whose elements are nondegenerate $k$-simplices for $k\geq 3$, each stored as an integer matrix with entries from the involutive set (\S\ref{subsec inv set}).
For 2-skeletal partial groups, the array is empty.
The first entry of the array is the set of nondegenerate $3$-simplices (should any exist), and the last element of the array is the set of nondegenerate $n$-simplices, where $n$ is the dimension of the partial group.
\end{remark}
There is redundancy here since the top row or superdiagonal of the $(k+1)$-by-$(k+1)$ matrix is enough to recover the remainder in the presence of the underlying BPG, but to get at the full symmetric set structure the matrices are preferable.

Now on to the algorithm. 
The general idea is similar to the previous case, and all extensions will be of the form $X \cup \lr{M_1} \cup \cdots \lr{M_p}$ for matrices $M_i \in \mat(X_1)_n$, but there are two limitations: all faces of the $M_i$ must be in $X$, and each $M_i$ must be suitable (\cref{def suitable}).

In this paragraph we'll write simplices in terms of their associated word via $\edgemap$.
For each nondegenerate\footnote{Recall from \cref{rmk implementation} that our implementation \emph{only} stores nondegenerate simplices.} simplex $x = (a_1,\dots, a_{n-1}) \in X_{n-1}$, we'd like to know if $x$ is extendable to a nondegenerate word $w=(a_1,\dots, a_{n-1},a_n)$ in an extension of $X$, where $a_n \in I = X_1\setminus \{0\}$.
We first check that $w$ is not in a list \texttt{blocked} of words which have already been processed (see below).
If it's not, we check that $d_0w = (a_2, \dots, a_n)$ and $d_iw = (a_1, \dots, a_{i+1}a_i, \dots, a_n)$ for $1\leq i \leq n-1$ are nondegenerate elements of $X_{n-1}$.
For $i=n-1$, this involves checking that $a_n a_{n-1}$ is actually defined in the underlying BPG.

Satisfying this is \emph{not} generally enough to guarantee that $w$ is the superdiagonal of a suitable matrix $M$.
The only issue arises when $n=3$, in which case we may have selected a word $(a_1,a_2,a_3)$ where in the underlying BPG we have $a_3(a_2a_1) \neq (a_3a_2)a_1$.
We explicitly test for this when deciding whether or not the length 3 word $w$ is admissible.
For $n > 3$, there are no such issues, as the equation $d_id_jw = d_{j-1}d_iw$ holds for $i < j$: for $i < j-1$ this is clear since the distant multiplications don't interact, and for $i=j-1$ this holds since $d_0w$ or $d_nw \in X_{n-1}$ implies $(a_i, a_{i+1}, a_{i+2}) \in X_3$, and hence $(a_{i+2}a_{i+1})a_i = a_{i+2}(a_{i+1}a_i)$.
The identity $d_id_jw = d_{j-1}d_iw$ implies that $w\in BP$ where $P$ is the underlying BPG of $X$ \cite[\S3]{HackneyLyndSalati:PGPG}, so has an associated suitable matrix.

If the word $w$ satisfies the checks above, then it is admissible.
We compute its entire $\Sigma_{n+1}$-orbit, which we add to the list \texttt{blocked} since we only want to consider one admissible word from each orbit. 

We have now constructed a set of admissible length $n$ words.
The $n$-dimensional extensions of $X$ correspond to the nonempty subsets $S$ of the set of ($\Sigma_{n+1}$-orbits of) admissible words.
But we mustn't forget Step~\ref{step 3}: we compute the orbits of the $\aut(X)$ action on the powerset of admissible words, and choose one representative from each orbit. 
We are left with a collection of powersets of admissible words, and it only remains for each of these to construct the relevant $n$-dimensional partial groups $Y$, by adding to $X$ the matrix form of each admissible word, along with its $\Sigma_{n+1}$-orbit. 
(We also compute $\aut(Y) \leq \aut(X)$ as the stabilizer of $S\in \wp(\mathtt{admissible})$ to pass along.)

\subsection{Termination} Given a BPG $P$ on a finite involutive set $I$, there are only finitely many possible extensions of the 2-dimensional partial group $\sk_2 BP$.
Indeed, at each step going from dimension $n-1$ to $n$, there are only finitely many words to consider (as all are elements of $I^n$), so only finitely many extensions of a fixed dimension $n$.
But \cref{hm thm} says that the maximal dimension we must consider is $|I|$.

\section{Summary of results}\label{sec results}

The main numerical results from the enumeration of partial groups of order up to 9 is given in \cref{table partial group counts} and \cref{table indecomp partial group counts}.
Here \emph{indecomposable} means in the sense of coproducts rather than products (i.e.\ it means that $X$ is nontrivial and does not decompose as $Y \vee Z$ for $Y,Z$ nontrivial).\footnote{These counts were obtained with a simple filter. Indecomposability of a partial group is detectable at the level of underlying BPGs. Given a BPG $P$, the algorithm creates a connectivity graph whose vertices are the nonzero elements of $P$, there is an edge between $a$ and $a^\dagger$, and whenever $a\cdot b = c$, there is a clique consisting of $\{a,b,c\}$. The BPG is indecomposable if and only if this graph is connected, which we test via union-find.}
The tables are organized around the isomorphism type of the underlying involutive set.
We first group by order and then by f, the number of free orbits of the involution.
Notice the second table has two fewer lines than the first, as the $(0,2)$ and $(0,4)$ types only contain decomposable partial groups.
We have omitted the trivial group from the first table.
The accompanying dataset \cite{partial_groups_data} contains a complete list of these partial groups (see \cref{sec data format}).

The partial groups with order one more than dimension are exactly the groups (\cref{hm thm}), and the results here are as expected based on the involution type. 
More surprising is the situation when dimension is two less than order.
In \cref{table indecomp partial group counts} we see that the counts within an order grouping in the last two columns agree.
This pattern holds for all orders, as proved in \cref{dim one less size}: such indecomposable partial groups are exactly the $(n{-}1)$-skeleta of the groups in the adjacent column.

\begin{table}
\centering
\caption{Indecomposable partial groups by order, free, dimension}
\label{table indecomp partial group counts}
\begin{tabular}{l c r r r r r r r r r}
\toprule
& \multicolumn{8}{c}{dimension} & \\
\cmidrule(lr){3-10}
order & f & 1 & 2 & 3 & 4 & 5 & 6 & 7 & 8 & total \\
\midrule
2 & 0 & 1 &   &   &   &   &   &   &   & 1 \\
\addlinespace[4pt]
3 & 1 & 1 & 1 &   &   &   &   &   &   & 2 \\
\addlinespace[4pt]
4 & 0 &   & 1 & 1 &   &   &   &   &   & 2 \\
  & 1 &   & 1 & 1 &   &   &   &   &   & 2 \\
\addlinespace[4pt]
5 & 1 &   & 2 &   &   &   &   &   &   & 2 \\
  & 2 &   & 3 & 1 & 1 &   &   &   &   & 5 \\
\addlinespace[4pt]
6 & 0 &   & 1 & 2 &   &   &   &   &   & 3 \\
  & 1 &   & 6 & 16 & 1 & 1 &   &   &   & 24 \\
  & 2 &   & 12 & 16 & 1 & 1 &   &   &   & 30 \\
\addlinespace[4pt]
7 & 0 &   & 2 & 12 &   &   &   &   &   & 14 \\
  & 1 &   & 14 & 50 &   &   &   &   &   & 64 \\
  & 2 &   & 54 & 71 & 1 &   &   &   &   & 126 \\
  & 3 &   & 51 & 37 & 5 & 1 & 1 &   &   & 95 \\
\addlinespace[4pt]
8 & 0 &   & 8 & 379 & 116 & 37 & 1 & 1 &   & 542 \\
  & 1 &   & 64 & 2250 & 585 & 112 & 1 & 1 &   & 3013 \\
  & 2 &   & 339 & 2256 & 433 & 87 & 1 & 1 &   & 3117 \\
  & 3 &   & 548 & 1426 & 408 & 57 & 2 & 2 &   & 2443 \\
\addlinespace[4pt]
9 & 0 &   & 17 & 768 &   &   &   &   &   & 785 \\
  & 1 &   & 338 & 10492 & 109 & 34 &   &   &   & 10973 \\
  & 2 &   & 3122 & 21443 & 648 & 98 &   &   &   & 25311 \\
  & 3 &   & 9705 & 20405 & 999 & 86 & 1 &   &   & 31196 \\
  & 4 &   & 4460 & 11794 & 12919 & 5387 & 45 & 2 & 2 & 34609 \\
\midrule
all  &  &  2 & 18749 & 71420 & 16226 & 5901 & 52 & 7 & 2 & 112359 \\
\bottomrule
\end{tabular}
\end{table}

To produce all of the partial groups up through order 9 takes approximately 2 minutes on a MacBook Air M4 with 32GB of memory.
Generating all BPGs of a given order is an easier task than generating all partial groups of that order.
We also were able to compute all 697,209 BPGs of order 10 in around 20 minutes.

\subsection{The order 10 case}
Computing all partial groups of order 10 is substantially harder than order 9.
For each of the 697,209 BPGs of order 10 one must enumerate all partial groups supported on it.
This task is incredibly difficult when the BPG is one of the two groups of order 10, due to many nontrivial multiplications.

There are 6,587,392 partial groups of order 10 whose underlying BPG is not a group.
This computation required a slightly more careful approach to reduce memory usage compared to what worked for smaller order.
Rather than materializing all extensions at each simplicial dimension simultaneously, the tree of extensions is traversed depth-first, keeping only O(depth) extensions in memory at once.
This calculation took roughly one hour, with the set of order 10 BPGs as an input.

Using this same modified algorithm, we also were able to enumerate all partial groups supported on a group of order 10.
\begin{itemize}[left=0pt]
\item 
There are 73,190,564 partial groups whose underlying BPG is the cyclic group $C_{10}$.
This enumeration took 11.6 days, nine of which were devoted to computing 42,517,128 extensions of $\sk_3 BC_{10}$.
\item 
There are 99,159,047 partial groups whose underlying BPG is the dihedral group $D_{10}$.
This enumeration took 16.8 days, of which 12.6 were devoted to computing 42,250,818 extensions of $\sk_3 BD_{10}$.
\end{itemize}

\subsection{Limitations for higher order}
A full enumeration in order 11 appears infeasible due to a combinatorial explosion.
Even the 2-dimensional case appears out of reach. 
We computed all 697,209 BPGs of order 10 in roughly 20 minutes, but order 11 is much harder: for each of the six involutive set types, the compatibility graph has 120--125 nodes and 5880--6470 edges (in contrast to ${\sim}85$ nodes/${\sim}2900$ edges for order 10), and the number of cliques (and hence BPGs) will be enormous.

\section{The degree of a partial group}\label{sec degree}
In this section we recall the \emph{(higher Segal) degree} of a spiny symmetric set $X$, following \cite[\S4.2]{HackneyLynd:HSSPG}, and explain an implementation for computing degree.
The degree invariant is derived from the higher Segal conditions for simplicial objects, which are based on the combinatorics of triangulations of cyclic polytopes and have connections to higher algebraic K-theory \cite{Poguntke:HSSAKT}, the Dold--Kan correspondence \cite{DyckerhoffJassoWalde}, and other areas; see \cite{Dyckerhoff:CPOC} for a survey.
In the context of partial groups, the higher Segal conditions are closely related to intersection patterns of closed subsets in closure spaces \cite{HackneyLynd:HSSPG}, and to a form of associativity \cite[\S3]{HackneyLyndSalati:EPG}.

\subsection{Computing degree}
When computing degree, it is helpful to work with the Bousfield--Segal map $\bous_n$ and the associated starry word complex $\bswords X$ \cite[Definition 2.10]{HackneyLynd:HSSPG}.
Level-wise, $\bswords X$ is given by $\bswords X_0 = X_0$, and for higher $n$ 
\[
  \bswords X_n = \{ (f_1, \dots, f_n) \mid d_1 f_1 = \dots = d_1 f_n \} \subseteq X_1^{\times n}.
\]
It is naturally a presheaf on the subcategory of $\fin$ on those maps which fix $0$ via $\alpha^*(f_1, \dots, f_n) \coloneq (f_{\alpha(1)}, \dots, f_{\alpha(n)})$ with the understanding that $f_0$ is an identity with the same source as the edges $f_1, \dots, f_n$.
In particular, for $1 \leq i \leq n$, the face map $d_i$ simply deletes entry $i$ and the degeneracy $s_i$ duplicates $f_i$, while $s_0$ inserts an identity at the beginning of the word.
The Bousfield--Segal maps combine into an injective map of presheaves $\bous \colon X \to \bswords X$.

\begin{definition}\label{def degree}
A spiny symmetric set $X$ has \emph{degree at most $k$} (for $k\geq 1$) if for each $n\geq k+1$, whenever $w\in \bswords X_n$ has the property that $d_{n-k} w, \dots, d_n w \in \bous_{n-1}(X_{n-1})$, then $w\in \bous_n(X_n)$.
The degree of $X$ is the least positive integer $k$ such that $X$ has degree at most $k$.
\end{definition}

See \cite{HackneyLynd:HSSPG} for the definition of degree for more general symmetric sets,\footnote{For spiny symmetric sets, Definition 3.18 of \cite{HackneyLynd:HSSPG} agrees with \cref{def degree}; this is a minor variation of \cite[Proposition 4.10]{HackneyLynd:HSSPG} as symmetries allow us to test only on the subset $\{n-k, \dots, n\}$.} which is based on an approach to the higher Segal conditions for simplicial objects from \cite{Walde:HSSHE}.
We note that a symmetric set has degree 1 if and only if it is isomorphic to the nerve of a groupoid.

For $w\in \bswords X_n$, we write $\overline w \in \bswords X_m$ for the word where we have deleted all identities and kept only the first instance of any non-identity $f_i$ appearing (if $w$ consists entirely of $\id_a$, then $\overline w$ is simply $a\in X_0 = \bswords X_0$).
There is a surjection $\sigma \colon [n] \twoheadrightarrow [m]$ such that $w = \sigma^* \overline w$, and by choosing a zero-preserving section $\delta \colon [m] \hookrightarrow [n]$ we have $\overline w = \delta^* w$.
We conclude that $w\in \bous_n(X_n)$ if and only if $\overline w \in \bous_m(X_m)$.
By our characterization of nondegeneracy, if $\overline w = \bous_m(x)$ then $x\in X_m$ is nondegenerate.
We call general elements of $\bswords X$ which do not contain identities or duplicates \emph{nondegenerate}.

\begin{proposition}\label{prop degree nondeg}
Suppose $X$ is a spiny symmetric set and $k\geq 1$.
Then $X$ has degree at most $k$ if for each $n\geq k+1$ and each \emph{nondegenerate} word $w\in \bswords X_n$ having $d_{n-k} w, \dots, d_n w \in \bous_{n-1}(X_{n-1})$, we also have $w\in \bous_n(X_n)$.
\end{proposition}
\begin{proof}
One direction of the equivalence is clear.
For the reverse, assume the condition from the proposition and suppose $w = (f_1, \dots, f_n) \in \bswords X_n$ has $d_{n-k} w, \dots, d_n w \in \bous_{n-1}(X_{n-1})$.
If for some $n-k \leq i \leq n$ the element $f_i$ appears elsewhere in $w$, or is an identity, then $w$ is degenerate on\footnote{More precisely, there exists a zero-preserving surjection $\sigma \colon [n] \to [n-1]$ with $w = \sigma^*d_i w$.} $d_iw$, hence is in $\bous_n(X_n)$.
Assuming this is not the case, then the reduced version $\overline w = (g_1, \dots, g_m)$ has the exact same tail: $g_{m-k} = f_{n-k}, \dots, g_m = f_n$.
Further, $d_i \overline w = \overline{d_{i-m+n} w}$ for $m-k \leq i \leq m$.
As $d_{i-m+n} w \in \bous_{n-1}(X_{n-1})$, we have $d_i \overline w = \overline{d_{i-m+n} w} \in \bous_{m-1}(X_{m-1})$, so by assumption $\overline w \in \bous_m(X_m)$.
This implies $w\in \bous_n(X_n)$.
\end{proof}

The upshot is that computing degree for a finite partial group is a finite process, and is one that works well with our existing implementation of finite partial groups (\cref{rmk implementation}).
Let us explain the basic algorithm.

For $n\geq k+1$, start with some nondegenerate element $x \in X_{n-1} \subseteq \bswords X_{n-1}$, which we write as $x = (a_1, \dots, a_{n-1})$.
We then consider all extensions of this word to $w = (a_1, \dots, a_{n-1}, a_n)$ where $a_n \in X_1 \setminus \{a_1, \dots, a_{n-1}, \id \}$.
We then look to see whether the faces $d_{n-k}w, \dots, d_{n-1}w$ are in $X_{n-1}$ (we do not need to check whether $d_n w = x$ is in $X_{n-1}$), and if they are, whether or not $w \in X_n$.
If some face is not in $X_{n-1}$, or if all faces are in $X_{n-1}$ and $w\in X_n$, then we continue the loop with the next element of $X_{n-1}$, otherwise we stop since we've established that $X$ does not satisfy the condition.
After processing all elements of $X_{n-1}$, we increase $n$ and continue.
This terminates since we are only working with nondegenerate elements.

\begin{remark}
There is a small optimization that we include: we only process those $x\in X_{n-1}$ such that the substrings $(a_1, \dots, a_{n-k-1})$ and $(a_{n-k}, \dots, a_{n-1})$ of $x$ are \emph{ordered}.
This is acceptable since the process from the previous paragraph is insensitive to the action by $\Sigma_{n-k-1} \times \Sigma_k$.
For the corpus of partial groups of order 9, this amounts to a 2.82-fold speedup (and will be more significant for larger partial groups).
\end{remark}

\subsection{Degree 2 partial groups are 2-coskeletal}\label{deg 2 partial groups}
This section records a result found during our experiments, which says that if a spiny symmetric set has degree at most 2, then it is equal to its 2-coskeleton. 
What does this mean?
The $n$-skeleton functor $\sk_n \colon \sym \to \sym$ admits a right adjoint, called the \emph{$n$-coskeleton} and denoted by $\cosk_n$. 
If $\iota \colon \fin_{\leq n} \to \fin$ is the inclusion from \cref{skeleta and dimension}, then $\cosk_n$ may be defined as $\iota_* \iota^*$, where $\iota^*$ is the restriction to $\fin_{\leq n}$, and $\iota_*$ is right Kan extension along $\iota^\op$.
There's no immediate formal reason why $\cosk_n$ preserves spininess, and is in fact not true when $n=1$.\footnote{For $X$ reduced, one sees $(\cosk_1 X)_2 = X_1^{\times 3}$ so $\cosk_1 X$ is spiny only when $X_1$ is a singleton.}
It will be helpful to have an alternate description of the coskeleton.

\begin{lemma}\label{lem n-coskel}
If $X$ is a spiny symmetric set and $n\geq 2$, then $\cosk_n X$ is isomorphic to the symmetric subset of $\mat(X_1)$ having $m$-simplices
\[
  \{ M \mid \alpha^*M \in \mu_X(X_n) \text{ for all } \alpha \colon [n] \to [m] \} \subseteq \mat(X_1)_m.
\]
In particular, $\cosk_n X$ is spiny.
\end{lemma}
\begin{proof}
Write $Y$ for the indicated symmetric subset of $\mat(X_1)$, and notice $\sk_n Y \cong \sk_n X$.
Since $n\geq 2$, \cref{spininess for mat subobjects} (identifying $Y_1$ with $X_1$) implies $Y$ is spiny.
Suppose $Z$ is an arbitrary symmetric set and $f \colon \sk_n Z \to X$ is a map.
The outer pentagon of the following commutes.
\[ \begin{tikzcd}[column sep=tiny, row sep=tiny]
& \sk_n Z  \drar{f} \dlar[hook] \\
Z\ar[dd,"\mu_Z"'] \drar[dashed] & & X\ar[dd,hook,"\mu_X"] \dlar[hook] \\
& Y \drar[hook] \\
\mat(Z_1) \ar[rr,"\mat(f_1)"'] & & \mat(X_1) 
\end{tikzcd} \]
For each $z\in Z_m$, the element $(\mat(f_1) \circ \mu_Z)(z)$ is in $Y_m$ since if $\alpha \colon [n] \to [m]$ is a map then $\alpha^*(\mat(f_1) \circ \mu_Z)(z) = (\mat(f_1) \circ \mu_Z)(\alpha^*z) = (\mu_X \circ f)(\alpha^*z) \in X_n$.
This defines a function $\hom(\sk_n Z, X) \to \hom(Z,Y)$.
Conversely given $Z\to Y$ we have $\sk_n Z \to \sk_n Y \cong \sk_n X \hookrightarrow X$. 
As maps into spiny symmetric sets are determined by what happens on vertices and edges, we have a natural isomorphism $\hom(\sk_n Z, X) \cong \hom(Z,Y)$.
Thus $Y \cong \cosk_n X$.
\end{proof}

\begin{remark}\label{B construction}
If $P$ is a BPG, then the partial group $BP$ is 2-coskeletal.
Moreover, if $X$ is a 2-coskeletal partial group with underlying BPG $P$, then $X\cong BP$.
\end{remark}

\begin{theorem}
Suppose $X$ is a spiny symmetric set and $k\geq 2$.
If $X$ has degree at most $k$, then $X$ is $k$-coskeletal.
\end{theorem}
\begin{proof}
Identify $X$ with a symmetric subset of $\mat(X_1)$, and let $\cosk_k X = Y \subseteq \mat(X_1)$ be the model for the $k$-coskeleton from \cref{lem n-coskel}.
As we always have $X\subseteq Y$, we inductively show the reverse inclusion.
We have $X_m = Y_m$ for $m \leq k$.
Suppose $n \geq k+1$ and $X_m = Y_m$ for $m < n$.
If $M\in Y_n$, then $d_{n-k}M, \dots, d_nM$ are all in $Y_{n-1} = X_{n-1}$.
Since $X$ has degree at most $k$, we have $M\in X_n$.
\end{proof}

\begin{corollary}\label{thm degree 2}
A spiny symmetric set of degree at most 2 is 2-coskeletal. \qed
\end{corollary}

\begin{remark}
One may compare this with results on coskeletality for 2-Segal simplicial sets, see \cite[\S4.1]{Stern:P2SC} and \cite[Corollary 1.7]{BOORS:2SSWC}.
\end{remark}

\begin{example}
The converse to \cref{thm degree 2} does not hold.
Among partial groups, there are seven counterexamples of order 6, and no counterexamples of smaller order.
The simplest counterexample is 2-dimensional, a BPG $P$ with $\sk_2 BP = BP$ where $P$ has underlying involutive set $\{0,\pm 1, \pm 2, 3\}$ and the following partial Cayley table.
\[
\begin{array}{r|rrrrr}
    & -2 &-1 & 1 & 2 & 3 \\
\hline
 -2 & -1 & 3 & 2 & 0 & ? \\
 -1 &  ? & ? & 0 &-2 & 2 \\
  1 &  2 & 0 & ? & 3 & ? \\
  2 &  0 &-2 & ? & 1 &-1 \\
  3 &  1 & ? &-2 & ? & 0 
\end{array}
\]
We have $(2 \cdot 2) \cdot 2 = 1 \cdot 2 = 3$, but $2 \cdot (2 \cdot 2) = 2 \cdot 1$ is not defined; thus $\deg BP > 2$ by \cite[Lemma 19]{HackneyLyndSalati:EPG}.
\end{example}

\appendix

\section{Indecomposable partial groups of order at most 5}\label{sec tiny order}

In this appendix we explicitly list all indecomposable partial groups of very small order, including the groups.
Recall that a partial group $X$ is indecomposable just when $X$ is nontrivial and if $X \cong Y \vee Z$, then $Y$ or $Z$ is trivial.
Completeness of this list follows from the algorithm in \cref{sec algorithm}.
We provide concrete descriptions, truncated partial Cayley tables (omitting multiplication by 0), sets of maximal subgroups (for non-groups), and minimal sets of generators (with orbit representatives chosen minimal with respect to column major order).
We also include the \emph{degree} (\cref{sec degree}) in \cref{table size at most 4} (see \cref{rmk degree calc} for details).

\begin{table}
\caption{Indecomposable partial groups of order at most 5}
\label{table size at most 4}
\begin{tabular}{c c c c c c r r r r}
\toprule
(free,fixed) & partial group & dim & gens & deg \\ 
\midrule
(0,1) & $BC_2$ & 1 & 1 &  1\\ 
(1,0) & $\sk_1(BC_3)$ & 1 & 1 & 2 \\ 
(1,0) & $BC_3$ & 2 & 1 &  1\\ 
(0,3) & $\sk_2(BV_4)$ & 2 & 1 & 3 \\ 
(0,3) & $BV_4$ & 3 & 1 &  1\\ 
(1,1) & $\sk_2(BC_4)$ & 2 & 1 &  3\\ 
(1,1) & $BC_4$ & 3 & 1 &  1\\ 
(1,2) & $BP_1$ & 2 & 1 & 2 \\ 
(1,2) & $BP_2$ & 2 & 2 & 2 \\ 
(2,0) & $\sk_2(BC_5)$ & 2 & 2 & 3 \\ 
(2,0) & $\sk_3(BC_5)$ & 3 & 1 & 4 \\ 
(2,0) & $BC_5$ & 4 & 1 &  1\\ 
(2,0) & $BP_3$ & 2 & 1 & 2 \\ 
(2,0) & $BP_4$ & 2 & 2 & 2 \\ 
\bottomrule
\end{tabular}
\end{table}

For order 2, there is only the cyclic group $BC_2$, which has involutive set $\{0,1\}$ and dimension 1. 
This has the Cayley table and generating matrix given below.
\[
\begin{array}{r|r}
  & 1 \\
  \hline
  1 & 0
\end{array}
\qquad
\begin{bmatrix}
  0 & 1 \\
  1 & 0
\end{bmatrix}
\]

At order 3, there are two indecomposable (binary) partial groups, of dimension 1 and 2, respectively: the free partial group $\mathfrak{F}^1$ on one generator and the cyclic group $BC_3$. 
Both have the same underlying involutive set $\{0,\pm1\}$, and we notice $\sk_1(BC_3) = \mathfrak{F}^1$.
The Cayley tables and generators are
\[
\begin{array}{r|rr}
  & -1 & 1 \\
  \hline
  -1 & ? & 0 \\
  1 & 0 & ?
\end{array}
\qquad
\begin{bmatrix}
  0 & 1 \\
  -1 & 0
\end{bmatrix}
\qquad
\begin{array}{r|rr}
  & -1 & 1 \\
  \hline
  -1 & 1 & 0 \\
  1 & 0 & -1
\end{array}
\qquad
\begin{bmatrix}
  0 & 1 & -1 \\
  -1 & 0 & 1 \\
  1 & -1 & 0
\end{bmatrix}
\]
and $\mathfrak{F}^1$ has only one subgroup (the trivial group).

There are exactly four indecomposable partial groups of order 4: the Klein 4-group $BV_4$, the cyclic group $BC_4$, and their 2-skeleta.
Each of these is generated by a single matrix.
For the first two in Table~\ref{table size at most 4}, we have involutive set $\{0,1,2,3\}$, and the following Cayley table and generators:
\[
\begin{array}{r|rrr}
  & 1 & 2 & 3 \\
  \hline
  1 & 0 & 3 & 2 \\
  2 & 3 & 0 & 1 \\
  3 & 2 & 1 & 0
\end{array}
\qquad
\begin{bmatrix}
  0 & 1 & 2 \\
  1 & 0 & 3 \\
  2 & 3 & 0
\end{bmatrix} \in \sk_2(BV_4) 
\qquad
\begin{bmatrix}
  0 & 1 & 2 & 3 \\
  1 & 0 & 3 & 2 \\
  2 & 3 & 0 & 1 \\
  3 & 2 & 1 & 0
\end{bmatrix}
\in BV_4.
\]
The non-group $\sk_2(BV_4)$ has 3 maximal subgroups, each a copy of $C_2$.
The next two entries in the table have involutive set $\{0,\pm1, 2\}$ and the following Cayley table and generators:
\[
\begin{array}{r|rrr}
  & -1 & 1 & 2 \\
  \hline
  -1 & 2 & 0 & 1 \\
  1 & 0 & 2 & -1 \\
  2 & 1 & -1 & 0
\end{array}
\quad
\begin{bmatrix}
  0 & 1 & -1 \\
  -1 & 0 & 2 \\
  1 & 2 & 0
\end{bmatrix}
\in \sk_2(BC_4)
\quad
\begin{bmatrix}
  0 & 1 & -1 & 2 \\
  -1 & 0 & 2 & 1 \\
  1 & 2 & 0 & -1 \\
  2 & -1 & 1 & 0
\end{bmatrix}
\in BC_4.
\]
The non-group $\sk_2(BC_4)$ has one maximal subgroup: $\{0,2\} \cong C_2$.

Turning to order 5, we see $BC_5$ and its skeleta, along with four other indecomposable partial groups, each associated to a BPG labeled $P_1$, $P_2$, $P_3$, $P_4$.

The partial group $BP_1$ (with $P_1 \cong A^{1,2}$ from \cref{ss abpgs}) has underlying involutive set $\{0,\pm1, 2, 3\}$, Cayley table and generating matrix
\[
\begin{array}{r|rrrr}
  & -1 & 1 & 2 & 3 \\
  \hline
  -1 & ? & 0 & 3 & ? \\
  1 & 0 & ? & ? & 2 \\
  2 & ? & 3 & 0 & 1 \\
  3 & 2 & ? & -1 & 0
\end{array}
\qquad
\begin{bmatrix}
  0 & 1 & 3 \\
  -1 & 0 & 2 \\
  3 & 2 & 0
\end{bmatrix},
\]
and two maximal subgroups $\{0,2\}$ and $\{0,3\}$.
Meanwhile, $BP_2$ has the same underlying involutive set, Cayley table and generating matrices
\[
\begin{array}{r|rrrr}
  & -1 & 1 & 2 & 3 \\
  \hline
  -1 & 1 & 0 & 3 & ? \\
  1 & 0 & -1 & ? & 2 \\
  2 & ? & 3 & 0 & 1 \\
  3 & 2 & ? & -1 & 0
\end{array}
\qquad
\begin{bmatrix}
  0 & 1 & -1 \\
  -1 & 0 & 1 \\
  1 & -1 & 0
\end{bmatrix},
\begin{bmatrix}
  0 & 1 & 3 \\
  -1 & 0 & 2 \\
  3 & 2 & 0
\end{bmatrix}
\]
and three maximal subgroups: $\{0,\pm 1\} \cong C_3$, $\{0,2\}$, and $\{0,3\}$.

Turning to skeleta of $BC_5$: we have Cayley table
\[
\begin{array}{r|rrrr}
  & -2 & -1 & 1 & 2 \\
  \hline
  -2 & -1 & 1 & 2 & 0 \\
  -1 & 1 & 2 & 0 & -2 \\
  1 & 2 & 0 & -2 & -1 \\
  2 & 0 & -2 & -1 & 1
\end{array}
\] on the underlying involutive set $\{0,\pm1,\pm2\}$. 
The generators for $\sk_2(BC_5)$ are
\[
\begin{bmatrix}
  0 & 2 & 1 \\
  -2 & 0 & 2 \\
  -1 & -2 & 0
\end{bmatrix},
\begin{bmatrix}
  0 & 2 & -1 \\
  -2 & 0 & 1 \\
  1 & -1 & 0
\end{bmatrix}
\]
and this BPG has no subgroup other than the trivial group.
The next partial group is $\sk_3(BC_5)$ which has a single generator
\[
\begin{bmatrix}
0 & 2 & 1 & -1 \\
-2 & 0 & 2 & 1 \\
-1 & -2 & 0 & 2 \\
1 & -1 & -2 & 0
\end{bmatrix},
\]
and no subgroup other than the trivial group.
We omit discussion of $BC_5.$

The partial group $BP_3$ (with $P_3 \cong A^{2,0}$ from \cref{ss abpgs}) has underlying involutive set $\{0,\pm1, \pm2\}$, Cayley table and generating matrix
\[
\begin{array}{r|rrrr}
  & -2 & -1 & 1 & 2 \\
  \hline
  -2 & -1 & ? & 2 & 0 \\
  -1 & ? & ? & 0 & -2 \\
  1 & 2 & 0 & ? & ? \\
  2 & 0 & -2 & ? & 1
\end{array}
\qquad
\begin{bmatrix}
  0 & 2 & 1 \\
  -2 & 0 & 2 \\
  -1 & -2 & 0
\end{bmatrix}
\]
and no nontrivial subgroup.
The partial group $BP_4$ has the same underlying involutive set, Cayley table and generating matrices
\[
\begin{array}{r|rrrr}
  & -2 & -1 & 1 & 2 \\
  \hline
  -2 & -1 & ? & 2 & 0 \\
  -1 & ? & 1 & 0 & -2 \\
  1 & 2 & 0 & -1 & ? \\
  2 & 0 & -2 & ? & 1
\end{array}
\qquad
\begin{bmatrix}
  0 & 2 & 1 \\
  -2 & 0 & 2 \\
  -1 & -2 & 0
\end{bmatrix},
\begin{bmatrix}
  0 & 1 & -1 \\
  -1 & 0 & 1 \\
  1 & -1 & 0
\end{bmatrix}
\]
and one maximal subgroup $\{0,\pm 1\} \cong C_3$.

\begin{remark}[Degree Calculations]\label{rmk degree calc}
The degree column in \cref{table size at most 4} was computed via the algorithm described in \cref{sec degree}, but one may reason about several of these entries as follows.
First, a partial group is the nerve of a group if and only if it has degree 1.
For the skeleta of groups, the degree follows from \cref{lem skeleton dimension} (coupled with \cref{hm thm} for $\dim(BG)$).
The only real calculation is for the four BPGs $P_i$ which can be shown to have degree at most 2 using \cref{prop degree nondeg}: in each case, there are no length 3 nondegenerate starry words $w$ with $d_1w, d_2w, d_3w\in BP_i$.
\end{remark}

\section{Partial groups of dimension two less than order}\label{sec appendix dim and order}

Recall from \cref{hm thm} that a partial group $X$ of order $n+1$ is a group if and only $\dim(X) = n$.
We establish the following related result, which is also true for order $\leq 5$ by the explicit list in \cref{sec tiny order}.

\begin{theorem}\label{dim one less size}
Suppose that $X$ is an indecomposable partial group of order $n+1 \geq 6$ and $\dim(X) = n-1$.
Then $X = \sk_{n-1} BG$ for a (unique) group $G$.
\end{theorem}
\begin{proof}
Consider a nondegenerate $(n{-}1)$-simplex $M = (f_{ij})$ in $X_{n-1}$, and let $a$ be the unique element of $X_1$ which does not appear in the first row of $M$.

We first observe that $a$ appears in $M$.
If $a\neq a^\dagger$, then since $a^\dagger$ is in the first row, $a$ is in the first column.
Suppose $a=a^\dagger$ does not appear among the elements of $M$.
Then the order of $\lr{M}$ is $n$ and its dimension is $n-1$, so $\lr{M}$ is a group.
Since $X$ is indecomposable we do not have $X = \lr{M} \vee \lr{a}$, which means there exists a nonidentity $b \in \lr{M}$ such that $ab$ is defined in the underlying BPG of $X$.
Since $ab \neq a$, we also have $ab \in \lr{M}$.
Then $a = (ab)b^\dagger$ is in the group $\lr{M}$, contrary to assumption.

Let $Y$ be the indecomposable partial group $\lr{M}$ and let $P$ be its underlying BPG.
As we saw in the previous paragraph, $X$ and $Y$ have the same set of edges.
Our immediate aim is to prove that $P$ is a group. 

Many products in $P$ are defined, for instance we have the multiplications $f_{0j}f_{i0} = f_{ij}$.
We show also that $a$ multiplies with many other elements: for each $i,j$ such that $f_{ij} = a$, let $\alpha_k \colon [2] \to [n-1]$ be given by $(0,1,2) \mapsto (k,i,j)$, and consider the set of matrices  $A_{ij} = \{ \alpha_k^*M \mid 0\leq k \leq n-1 \}$.
This is the set of 2-simplices
\[
\alpha_k^*M = 
\begin{bmatrix}
0 & f_{ki} & f_{kj} \\
f_{ik} & 0 & a \\
f_{jk} & a^\dagger & 0
\end{bmatrix} \qquad 0\leq k \leq n-1,
\]
so $af_{ki}$ is defined for all $0\leq k \leq n-1$.
But the $f_{ki}$ are distinct, so $ab$ is defined for at least $n$ elements of $Y_1$.

We next show that an element $x\in Y_1$ can be missing from at most two rows of $M$.
If $x$ is not in row $i$, then there does not exist a $j$ with $x = f_{ij} = f_{0j}f_{i0}$.
This implies either that $xf_{0i}$ is undefined, or that it is defined and equal to $a$.
There is at most one $i$ such that the latter holds.
If $xf_{0i}$ is undefined then $x$ is not in the first row or $f_{0i}$ is not in the first column, i.e.\ $x=a$ or $f_{0i} = a^\dagger$. 
These cases are mutually exclusive since $aa^\dagger$ is defined.
There is of course at most one $i$ such that $f_{0i} = a^\dagger$, and by the previous paragraph there is at most one $i$ such that $af_{0i}$ is undefined.
Thus we have indeed shown there are at most two rows not containing $x$.

Given $x,y \in Y_1$, there are at most four rows which do not contain the set $\{x,y\}$. 
Since $n\geq 5$, there is a row containing $x$ and $y$, hence 2-simplices
\[
\begin{bmatrix}
0 & x & y \\
x^\dagger & 0 & z \\
y^\dagger & z^\dagger & 0 
\end{bmatrix}
\rightsquigarrow 
\begin{bmatrix}
0 & x^\dagger & z \\
x & 0 & y \\
z^\dagger & y^\dagger & 0
\end{bmatrix}.
\]
In particular, the operation $(x,y) \mapsto yx^\dagger$ is totally defined, so multiplication is totally defined in $P$.

For a totally defined multiplication, the argument in the paragraph above gives that each element $x$ is missing from at most one row.
Thus the assignment $[n-1] \to Y_1 \setminus \{0\}$ sending $i$ to the element which row $i$ is missing is injective, hence bijective, and each $x\neq 0$ appears in exactly $n-1$ rows.

We next prove associativity $(xy)z = x(yz)$ in $P$. 
At least $n-2 \geq 3$ columns of $M$ contain both $y$ and $yz$ --- pick such a column $j$ such that $x$ is in row $j$. 
We write $y = f_{ij}$, $yz = f_{kj}$, and $x = f_{j\ell}$ for some $i,k,\ell \in [n-1]$.
Define $\alpha \colon [3] \to [n-1]$ sending $0,1,2,3$ to $k,i,j,\ell$.
The 3-simplex $\alpha^*M$ has superdiagonal $(z,y,x)$, and we conclude $(xy)z = x(yz)$.

We have now established that $P$ is a group, which we subsequently write as $G$. 
It remains to show $\sk_{n-1} BG \subseteq Y$, namely that every nondegenerate $n-1$ simplex of $BG$ is contained in $Y$.
Via the Bousfield-Segal map $\bous_{n-1}$, these are in bijection with ordered lists of distinct nonidentity elements $(g_1, \dots, g_{n-1})$.
There is a unique row $j$ of $M$ whose nonidentity entries form the set $\{g_1, \dots, g_{n-1}\}$ and $M' = (0j)^*M$ has the tail of its first row exactly that set of entries.
Apply an automorphism of $[n-1]$ which fixes zero to $M'$ to put the $g_i$ in the first row into the correct order.
We have then established that each nondegenerate $n-1$ simplex of $BG$ is in the $\Sigma_n$-orbit of $M$, hence in $\lr{M} = Y$.

To conclude, we have inclusions $\sk_{n-1} BG \subseteq Y \subseteq X \subseteq BG$. 
Taking $(n{-}1)$-skeleta we see $\sk_{n-1} BG = \sk_{n-1} X$, which is equal to $X$.
\end{proof}

\begin{proposition}\label{prop dim one less size decomp}
Suppose $X$ is a decomposable partial group of order $n+1$ and $\dim(X) = n-1$. 
Then $X \cong BG \vee BC_2$ for a (unique) group $G$.
\end{proposition}
\begin{proof}
Write $X = Y \vee Z$ with $Y,Z$ nontrivial partial groups and $\dim(Y) = n-1$.
By the first part of \cref{hm thm}, $Y$ cannot have fewer than $n-1$ nonidentity elements, so it must have exactly $n-1$.
By the second part, $Y$ is a group of order $n$.
As $Z$ has a single nonidentity element, $Z \cong BC_2$.
\end{proof}

The following concerns the invariant from \cref{sec degree}.

\begin{corollary}
Let $n\geq 1$, and suppose $X$ is a partial group of order $n+1$ and dimension $n-1$.
If $X$ is indecomposable then $\deg(X) = n$ and if $X$ is decomposable then $\deg(X) = 2$.
\end{corollary}
\begin{proof}
This follows from \cref{dim one less size}, \cref{prop dim one less size decomp}, and \cref{lem skeleton dimension} below.
\end{proof}

The statement of the following lemma is a variation of \cite[Example 3.19]{HackneyLynd:HSSPG}.

\begin{lemma}\label{lem skeleton dimension}
Suppose $X$ is a spiny symmetric set with $\dim(X) \geq n$. 
Then $\deg(\sk_{n-1}X) = n$.
\end{lemma}
\begin{proof}
By Theorem 9.6 of \cite{HackneyLynd:HSSPG}, $\deg(\sk_{n-1}X) \leq \dim(\sk_{n-1} X) + 1 = n$.
Suppose $x\in X_n$ is nondegenerate and $k=n-1$.
Then $d_i x \in (\sk_{n-1} X)_{n-1}$ for $n-k = 1 \leq i \leq n$.
But $x\notin (\sk_{n-1} X)_n$, so $\sk_{n-1} X$ does not have degree at most $k = n-1$.
\end{proof}

\section{Data format and usage}\label{sec data format}
A companion to this paper is a dataset consisting of all partial groups and BPGs of order at most 10 \cite{partial_groups_data}.
These are both stored in the JSON Lines format (see \url{https://jsonlines.org/}).
A reference importer (in Python) is included alongside the dataset.
The associated Julia project \cite{partialgroups_jl} also contains facilities for importing from the dataset.

\subsection{The BPG file format}
The binary partial groups are split into two files: one for all BPGs of order up to 9, and another, much larger file containing the BPGs of order 10.
Each line represents a BPG, and consists of four fields:
\begin{description}
\item[\texttt{free}] the number of free orbits in the underlying involutive set.
\item[\texttt{fixed}] the number of nonidentity fixed points in the underlying involutive set.
\item[\texttt{bpg\char`_hash}] a string which gives a stable short identifier (derived from SHA-256).
\item[\texttt{mults}] an array of nontrivial generating multiplications for the BPG.
\end{description}

The (\texttt{free},\texttt{fixed}) pair determines the underlying involutive set (see \cref{subsec inv set}).
Each generating multiplication in \texttt{mults} takes the form \texttt{[a,b,c]}, which should be interpreted as specifying $\mathtt{a}\cdot \mathtt{b} = \mathtt{c}$.
The full BPG may be reconstructed as the union of the 1-skeleton (specified by the involutive set) together with a collection of atomic BPGs, one for each element of \texttt{mults}.
As examples:
\begin{itemize}[left=0pt]
\item \verb|[1,0,"f1_x0_200428d33786",[]]| is the representation for the free partial group on one generator.
\item \verb|[1,0,"f1_x0_b50145ab2deb",[[1,1,-1]]]| is the representation for $C_3$ (considered as a BPG). 
\item \verb|[1,2,"f1_x2_7d59313238d9",[[1,1,-1],[1,3,2]]]| is the BPG called $P_2$ in \cref{sec tiny order}.
\end{itemize}
We warn that the algorithm for computing \texttt{bpg\char`_hash} is not isomorphism-invariant.

\subsection{The partial group file format}
The partial groups are spread over a number of files, one for each (\texttt{free},\texttt{fixed}) pair.\footnote{With variations for the (2,5) and (4,1) cases in order 10, which are each split over two files.}
The format for these is even simpler -- it consists of a \verb|bpg_hash| string identifying the underlying BPG, followed by a list of generating matrices for the partial group.
The set of generating matrices completely determines the partial group, as a symmetric subset of $\mat(\mathbb{Z})$ (see \cref{sec matrices to SSS}).
\begin{itemize}[left=0pt]
\item \verb|["f0_x3_04f444b95639",[[[0,1,2,3],[1,0,3,2],[2,3,0,1],[3,2,1,0]]]]|  is the partial group $BV_4$ (see \cref{sec tiny order}).
\item \verb|["f1_x1_f8a813dd9579",[[[0,1,-1],[-1,0,1],[1,-1,0]],[[0,2],[2,0]]]]| is the decomposable partial group $BC_3 \vee BC_2$.
\end{itemize}
The partial group still needs to be regenerated from the list of generating matrices $M$: the nondegenerate simplices are all of the form $\alpha^*M$ where $\alpha$ is injective.
We note that \verb|bpg_hash| is not necessary for reconstructing the partial group, but is useful for finding which partial groups have the same underlying BPG, without doing the full reconstruction.

\bibliographystyle{alpha}
\bibliography{enum_pg}
\end{document}